\documentclass[12pt,a4paper]{amsart}

\usepackage{amssymb}
\usepackage{amsmath}
\usepackage{bbm}
\usepackage{amsthm}
\usepackage{enumerate}
\usepackage{hyperref}
\hypersetup{colorlinks=true,allcolors=blue}
\usepackage{orcidlink}

\newtheorem{proposition}{Proposition}
\newtheorem{lemma}[proposition]{Lemma}
\newtheorem{corollary}[proposition]{Corollary}
\newtheorem{theorem}[proposition]{Theorem}

\newtheorem{remark}[proposition]{Remark}

\theoremstyle{definition}
\newtheorem{example}[proposition]{Example}
\newtheorem{definition}[proposition]{Definition}

\newcommand{\Z}{\mathbb{Z}}
\newcommand{\F}{{\mathbb F}}
\newcommand{\RR}{{\mathbb R}}
\newcommand{\mm}{\mathcal{M}}
\newcommand{\CC}{\mathcal{C}}
\newcommand{\rng}[1]{\langle #1 \rangle}

\begin{document}

\title{Compressed commuting graphs of matrix rings}

\author[I.-V. Boroja]{Ivan-Vanja Boroja \orcidlink{0000-0003-1836-0381}}
\address{University of Banja Luka, Faculty of Architecture, Civil Engineering and Geodesy, Bosnia and Herzegovina}
\email{ivan-vanja.boroja@etf.unibl.org}

\author[H. R. Dorbidi]{Hamid Reza Dorbidi \orcidlink{0000-0002-8736-0476}
}
\address{Department of Mathematics, Faculty of Science, University of Jiroft, Iran}
\email{hr\_dorbidi@ujiroft.ac.ir}

\author[D. Kokol Bukov\v sek]{Damjana Kokol Bukov\v{s}ek \orcidlink{0000-0002-0098-6784}}
\address{University of Ljubljana, School of Economics and Business, and Institute of Mathematics, Physics and Mechanics, Ljubljana, Slovenia}
\email{damjana.kokol.bukovsek@ef.uni-lj.si}

\author[N. Stopar]{Nik Stopar \orcidlink{0000-0002-0004-4957}}
\address{University of Ljubljana, Faculty of Civil and Geodetic Engineering, and Institute of Mathematics, Physics and Mechanics, Ljubljana, Slovenia}
\email{nik.stopar@fgg.uni-lj.si}

\begin{abstract}
In this paper we introduce compressed commuting graph of rings. It can be seen as a compression of the standard commuting graph (with the central elements added) where we identify the vertices that generate the same subring. The compression is chosen in such a way that it induces a functor from the category of rings to the category of graphs, which means that our graph takes into account not only the commutativity relation in the ring, but also the commutativity relation in all of its homomorphic images. 
Furthermore, we show that this compression is best possible for matrix algebras over finite fields, i.e., it compresses as much as possible while still inducing a functor. We compute the compressed commuting graphs of finite fields and rings of $2 \times 2$ matrices over finite fields.
\end{abstract}

\thanks{The support by the bilateral grant BI-BA/24-25-024 of the ARIS (Slovenian Research and Innovation Agency) is gratefully acknowledged. Damjana Kokol Bukovšek and Nik Stopar acknowledge financial support from the ARIS (Slovenian Research and Innovation Agency, research core funding No. P1-0222). Ivan-Vanja Boroja acknowledges financial support from the Municipality of Mrkonjić Grad.}
\keywords{Commuting graph, compressed commuting graph, matrix ring, finite field}
\subjclass[2020]{05C25, 15A27, 18A99}

\maketitle

\section{Introduction}

The essence of the commutativity relation on a given algebraic structure $A$ equipped with a product operation (e.g. ring, group, semigroup, etc.) is captured in its commuting graph $\Gamma(A)$. By definition, this is a simple graph whose vertices are all non-central elements of $A$ and where two distinct vertices $a, b$ are connected if they commute in $A$, i.e., if $ab = ba$. As far as we know, the commuting graph was first introduced for groups in \cite{BrFo} in an early attempt towards classification of simple finite groups. It was later extended to rings in \cite{AkGhHaMo04} and other algebraic structures \cite{ArKiKo, WaXi}.

Commuting graphs have seen a lot of attention in the last two decades. Various authors have studied their connectedness \cite{AkBiMo, AkGhHaMo04, DoD19}, diameter \cite{AkMoRaRa06, DoKoKu18, Sh16}, spectra and energies \cite{DuNa17, DuNa21}, and other properties. An important question in the theory is the isomorphism problem. It asks about the conditions under which an isomorphism of commuting graphs implies an isomorphism of underlying structures. In \cite{AkGhHaMo04} the authors conjectured that  the commuting graph is able to distinguish a ring of $n \times n$ matrices over a finite field among all finite unital rings. The conjecture was proved correct for $n=2$ and $n=3$ in \cite{Mo10,DoMa24} and for $n=2^k3^l$ with $k \ge 1$ in \cite{DoH19}, but it is still open for general $n$. In the category of groups, the commuting graph is able to characterize finite simple nonabelian groups, \cite{SoWo}. The commuting graph can distinguish among algebras of bounded operators on a complex Hilbert space, i.e. it can calculate the dimension of the underlying vector space, \cite{Ku18}.

Graphs play an important role in understanding algebraic structures also in other ways. Other types of graphs that are extensively studied are zero-divisor graph \cite{AnLi}, total graph \cite{AnBa}, power graph \cite{AbKeCh}, nilpotent graph \cite{AbZa10, DaNo15},
solvable graph \cite{BhNoNa20, BhNoNa22}, super graph \cite{ArCaRa24, ArCaNaSe22}, etc.

In \cite{Mu02} Mulay introduced a compressed zero-divisor graph of a ring. The idea of compression is to combine certain vertices of the graph into a single vertex, in order to make the graph smaller, and thus more manageable, while still keeping the essence of the property that the graph is capturing.
For this to work, we can only combine vertices that are indistinguishable in the non-compressed graph. 
While early examples of compressed zero-divisor graphs behave chaotically when ring homomorphisms are considered, a new, categorical approach was recently developed in \cite{DuJeSt21a,DuJeSt21b} to resolve this issue. In fact, for a unital ring $R$, the authors define a compressed zero-divisor graph $\Theta(R)$ using compression that induces a functor.
This new compressed zero-divisor graph takes into account not only the zero-divisor structure of the ring in question but also the zero-divisor structure of all possible homomorphic images of that ring.
The clear benefit of this approach is that the graph $\Theta(R)$ better captures the structural properties of rings.
In particular, we can recognize from the compressed zero-divisor graph of a finite unital commutative ring $K$ whether the ring $K$ is local, and whether it is a principal ideal ring.
Furthermore, the compressed zero-divisor graph of a direct product of finite unital rings is a tensor product of their compressed zero-divisor graphs.
See \cite{DuJeSt21a,DuJeSt21b} for more details.

The aim of this paper is to introduce
a compressed commuting graph
of a (unital) ring
$R$
and study its structure when $R$ is matrix ring over a finite field.
The motivation for this is as follows.
As it was the case for zero-divisor graph, we hope that this newly developed graph will shed new light on the commutativity relation in rings.
As mentioned above, the isomorphism problem for commuting graph of matrix rings is still open in general.
Since the compression gives additional information about the ring (in particular, we may encode how many elements get compressed into a single vertex into the graph, see Definition~\ref{def:weighted}), this new graph might help solve the isomorphism problem for matrix rings in the future.

Similarly as in the case of the compressed zero-divisor graph we want to use a compression that will induce a functor from the category of (unital) rings
to the category of undirected simple graphs with added loops, so that the graph will behave nicely with respect to ring homomorphisms.
For a (unital) ring $R$, the vertices of our (unital) compressed commuting graph $\Lambda(R)$ (resp. $\Lambda^1(R)$ in the unital case) are equivalence classes of elements of $R$, with respect to the equivalence relation defined by $a \sim b$ if and only if $a$ and $b$ individually generate the same (unital) subring of $R$.
Two vertices are connected by an edge if their respective representatives commute in $R$ (see Definitions~\ref{def:CCG} and~\ref{def:UCCG} for details).
Unlike in the classical commuting graph $\Gamma(R)$, we do not exclude the center of $R$ from the graph.
This means that we may also consider the compressed commuting graph of a commutative ring $K$. This graph is not trivial and it gives nontrivial information about the ring $K$.
In particular, for a general ring $R$, every vertex of $\Lambda(R)$ (resp. $\Lambda^1(R)$) corresponds to a (unital) subring of $R$ generated by one element, so the (unital) compressed commuting graph of $R$ gives information about the set of such subrings of $R$.
In fact, we can equivalently define (unital) compressed commuting graph of a ring $R$ to be the graph with vertex set equal to the set of all (unital) subrings of $R$ generated by one element, see Remark~\ref{rem:def}. Consequently, the size of the graph is equal to the number of (unital) subrings generated by one element.
The chosen relation $\sim$ turns out to be the best possible in the case of matrix algebras over finite fields in the sense that it is the coarsest equivalence relation that induces a functor (see Theorem~\ref{thm:best1} and Remark~\ref{rem:best} for details).

The paper is structured as follows. In Section~\ref{sec:def}, which constitutes a large part of our paper, we define a compressed commuting graph of a ring $R$ and a unital compressed commuting graph of a unital ring $R$.
We give a connection between both graphs and study their properties from a categorical point of view.
In particular, we show that, from a categorical perspective, the chosen compression is optimal in the case of matrix rings over finite fields.
In Section~\ref{sec:fields} we obtain some auxiliary results, which we need in Section~\ref{sec:matrices}. In particular, we describe the (unital) compressed commuting graphs of finite fields, the graph of a direct product of two finite fields, and the graph of a specific factor ring of the polynomial ring.
In Section~\ref{sec:matrices} we initiate the study of compressed commuting graphs of matrix rings. We describe the (unital) compressed commuting graph of the ring $R$ of $2 \times 2$ matrices over an arbitrary finite field $GF(p^n)$. As a corollary, we determine the asymptotic behaviour of the number of (unital) subrings of $R$ generated by one element as $n$ tends to infinity while the prime $p$ is fixed.

\section{Definition and properties}\label{sec:def}

Let $R$ be a general ring, possibly nonunital. A subring of $R$ generated by an element $a \in R$ will be denoted by $\rng{a}$, i.e. $\rng{a}=\{q(a) ~|~ q \in \Z[x],\, q(0)=0\}$, where $\Z[x]$ denotes the ring of polynomials with integer coefficients.
The \emph{centralizer} of an element $a \in R$ will be denoted by $\CC(a)$, i.e. $\CC(a)=\{x \in R ~|~ ax=xa\}$.
We introduce an equivalence relation $\sim$ on $R$ defined by $a \sim b$ if and only if $\rng{a}=\rng{b}$, and denote the equivalence class of an element $a \in R$ with respect to relation $\sim$ by $[a]$.
By definition $[a]$ consists of all single generators of the ring $\rng{a}$.

\begin{definition}\label{def:CCG}
A \emph{compressed commuting graph} of a ring $R$ is an undirected graph $\Lambda(R)$ whose vertex set is the set of all equivalence classes of elements of $R$ with respect to relation $\sim$ and there is an edge between $[a]$ and $[b]$ if and only if $ab=ba$.
\end{definition}

Note that edges in $\Lambda(R)$ are well defined. Indeed, if $[a]=[a']$, $[b]=[b']$, and $ab=ba$, then $a' \in \rng{a}$ and $b' \in \rng{b}$, hence, $a',b' \in \rng{a,b}$, the subring generated by $a$ and $b$.
But since $a$ and $b$ commute, $\rng{a,b}$ is a commutative ring, hence $a'$ and $b'$ commute as well.
It should be remarked that central elements of $R$ are not excluded from the graph $\Lambda(R)$ like in the usual commuting graph $\Gamma(R)$.
Furthermore, loops are allowed in $\Lambda(R)$, in fact, every vertex of $\Lambda(R)$ has a single loop on it.

\begin{remark}\label{rem:def}
As already mentioned, each vertex of $\Lambda(R)$ corresponds to a subring of $R$ generated by one element. This means that we could equivalently define the compressed commuting graph of $R$ as an undirected graph whose vertex set is the set
$$V(\Lambda(R))=\{\rng{a} \mid a \in R\}$$
of all subrings of $R$ generated by one element, and vertices $\rng{a}$ and $\rng{b}$ are connected by an edge if and only if $ab=ba$.
\end{remark}

In what follows we extend the mapping $\Lambda$ to a functor $\Lambda$. To do this we will need some basic notions from category theory, see for example \cite{MacL98}. In particular, recall that a functor $F$ between two categories $\bf A$ and $\bf B$ is a mapping that associates to each object $X$ in $\bf A$ an object $F(X)$ in $\bf B$, and to each morphism $f \colon X \to Y$ in $\bf A$ a morphism $F(f) \colon F(X) \to F(Y)$ in $\bf B$, such that $F(\text{id}_X)=\text{id}_{F(X)}$ and $F(f\circ g)=F(f) \circ F(g)$ for all objects $X$ in $\bf A$ and all morphisms $g \colon X \to Y$, $f \colon Y \to Z$ in $\bf A$.
Recall also that $\mathbf{Ring}$ is a category, whose objects are all (possibly nonunital) rings and morphisms are ring homomorphisms. Furthermore, $\mathbf{Graph}$ is a category, whose objects are undirected simple graphs that allow loops and morphisms are graph homomorphisms.

Note that the mapping $\Lambda$ associates to each ring $R$ a graph $\Lambda(R)$. Now we extend this mapping to a functor $\Lambda$ from the category $\mathbf{Ring}$ to the category $\mathbf{Graph}$. 
For a ring morphism $f \colon R \to S$ we define a graph morphism $\Lambda(f) \colon \Lambda(R) \to \Lambda(S)$ by $\Lambda(f)([r])=[f(r)]$.
We need to verify that the map $\Lambda(f)$ is well defined. If $[r]=[r']$ then there exist polynomials $p,q \in \Z[x]$ with $p(0)=q(0)=0$ such that $r'=p(r)$ and $r=q(r')$.
Hence, $f(r')=p(f(r))$ and $f(r)=q(f(r'))$ and consequently $[f(r)]=[f(r')]$.
Furthermore, $\Lambda(f)$ maps connected vertices to connected vertices since $ab=ba$ implies $f(a)f(b)=f(b)f(a)$.
So the map $\Lambda(f)$ is indeed a graph morphism.

Clearly, $\Lambda(\operatorname{id}_R)=\operatorname{id}_{\Lambda(R)}$ for any ring $R$ and $\Lambda(f\circ g)=\Lambda(f) \circ \Lambda(g)$ for all ring morphisms $f \colon S \to T$ and $g \colon R \to S$. This proves the following.

\begin{proposition}\label{prop:functor}
    The mapping $\Lambda \colon \mathbf{Ring} \to \mathbf{Graph}$ that maps a ring $R$ to the graph $\Lambda(R)$ and a ring morphism $f$ to the graph morphism $\Lambda(f)$ is a functor.
\end{proposition}

Note that the fact that edges in the graph $\Lambda(R)$ are well defined implies that the equivalence relation $\sim$ has the following property
\begin{equation*}
    a \sim b \ \Rightarrow \ \CC(a)=\CC(b).
\end{equation*}

Next we define a unital version of the above functor for unital rings.
Let $R$ be a unital ring with identity element $1$.
Let $\rng{a}_1$ denote the subring of $R$ generated by $1$ and $a \in R$ (we will also call $\rng{a}_1$ a \emph{unital subring} of $R$ generated by $a$).
Note that $\rng{a}_1=\{q(a) ~|~ q \in \Z[x]\}$, where in the evaluation $q(a)$ the constant term of $q$ is multiplied by the identity element $1 \in R$.
Now define an equivalence relation on $R$ by $a \sim_1 b$ if and only if $\rng{a}_1=\rng{b}_1$ and denote the equivalence class of an element $a \in R$ by $[a]_1$.

\begin{definition}\label{def:UCCG}
    A \emph{unital compressed commuting graph} of a unital ring $R$ is an undirected graph $\Lambda^1(R)$ whose vertex set is the set of all equivalence classes of elements of $R$ with respect to relation $\sim_1$ and there is an edge between $[a]_1$ and $[b]_1$ of and only if $ab=ba$.
\end{definition}

Similarly as in the previous case, the graph $\Lambda^1(R)$ may be defined  as an undirected graph whose vertex set is the set
$$V(\Lambda^1(R))=\{\rng{a}_1 \mid a \in R\}$$
of all unital subrings of $R$ generated by one element, and vertices $\rng{a}_1$ and $\rng{b}_1$ are connected by an edge if and only if $ab=ba$.
The mapping $\Lambda^1$ can be extended to a functor $\Lambda^1$ from the category $\mathbf{Ring^1}$ of unital rings and unital ring homomorphisms to the category $\mathbf{Graph}$.
For this to work we need to consider the zero ring $R=0$ as a unital ring with $1=0$.
In particular, for any unital ring morphism $f \colon R \to S$, where $R$ and $S$ are unital rings, we define $\Lambda^1(f) \colon \Lambda^1(R) \to \Lambda^1(S)$ again by $\Lambda^1(f)([r]_1)=[f(r)]_1$.
The fact that $f$ is unital implies that $\Lambda^1(f)$ is well defined.

\begin{proposition}\label{prop:functor1}
    The mapping $\Lambda^1 \colon \mathbf{Ring^1} \to \mathbf{Graph}$ that maps a unital ring $R$ to the graph $\Lambda^1(R)$ and a unital ring morphism $f$ to the graph morphism $\Lambda^1(f)$ is a functor.
\end{proposition}

For a general ring $R$, let $m$ be the least positive integer, if it exists, such that $mR=0$. Otherwise let $m=0$. The number $m$ is called the \emph{characteristic} of ring $R$. We may adjoin an identity element to $R$ to form a unital ring $R^1$ as follows.
Note that $R$ is a left $\Z_m$-module, where $\Z_0=\Z$ if $m=0$.
Equip the set $R^1=\Z_m \times R$ with componentwise addition and define multiplication in $R^1$ by $(k,a)(n,b)=(kn,na+kb+ab)$.
Then $R^1$ is a unital ring with identity element $(1,0)$.
If $m=p$ is prime then $R^1$ is also an algebra over $\Z_p \cong GF(p)$ with scalar multiplication defined componentwise.
The ring $R$ is an ideal of $R^1$ with the canonical embedding $i \colon R \to R^1$ given by $i(r)=(0,r)$.
We have the following connection between functors $\Lambda$ and $\Lambda^1$.

\begin{proposition}\label{prop:iso}
    For any ring $R$ (unital or nonunital) we have $\Lambda(R) \cong \Lambda^1(R^1)$, where the isomorphism $\hat{\imath} \colon \Lambda(R) \to \Lambda^1(R^1)$ is given by $\hat{\imath}([a])=[i(a)]_1$.
\end{proposition}

\begin{proof}
    Let $i \colon R \to R^1=(\Z_m \times R,+,\cdot)$ be the canonical embedding.
    Note that for any $(k,a) \in R^1$ the ring $\rng{(k,a)}_1$ is generated by $(k,a)$ and $(1,0)$, thus also by $(0,a)=i(a)$ and $1=(1,0) \in R^1$.
    Hence, it contains $\Z\cdot 1 +\rng{i(a)}=\Z_m \times \rng{i(a)} \subseteq R^1$. But the latter is clearly a unital subring of $R^1$, therefore $\rng{(k,a)}_1=\Z_m \times \rng{i(a)}=\Z\cdot 1 +i(\rng{a})=\rng{i(a)}_1$.
    This implies that the map $\hat{\imath} \colon \Lambda(R) \to \Lambda^1(R^1)$ given by $\hat{\imath}([a])=[i(a)]_1$ is a well defined bijection on the sets of vertices.
    Furthermore, for any $a,b \in R$ we have $ab=ba$ if and only if $i(a)i(b)=i(b)i(a)$. Hence, $\hat{\imath}$ is a graph isomorphism.
\end{proof}

The following gives the motivation for choosing the particular equivalence relation in the above definitions.
It implies that, at least on finite unital algebras, $\sim_1$ is the coarsest relation that still induces a functor, so it is the best possible for our purpose.

\begin{theorem}\label{thm:best1}
    For each unital ring $R$ let $\approx_R$ be an equivalence relation on $R$ such that the family $\{\approx_R |~ R \text{ a unital ring}\,\}$ induces a well defined functor $F \colon \mathbf{Ring^1} \to \mathbf{Graph}$ in the following way:
    \begin{enumerate}[$(i)$]
        \item For each unital ring $R$ the vertices of $F(R)$ are equivalence classes $[r]_{\approx_R}$ of elements of $R$ with respect to $\approx_R$ and there is an edge between $[a]_{\approx_R}$ and $[b]_{\approx_R}$ if and only if $ab=ba$.
        \item For each unital ring morphism $f \colon R \to S$, where $R$ and $S$ are unital rings, the graph morphism $F(f) \colon F(R) \to F(S)$ is given by
        $$F(f)([r]_{\approx_R})=[f(r)]_{\approx_S}$$ for all $r \in R$.
    \end{enumerate}
    Then for any finite unital algebra $A$ and for any $a,b \in A$ the condition $a \approx_A b$ implies $a \sim_1 b$. 
\end{theorem}

\begin{proof}
    Let $A$ be a finite unital algebra. Since it is finite, it is an algebra over a finite field $\F$,  
    the characteristic of $\F$ is a prime $p$, and its prime field is $GF(p)$.
    Thus, we may consider $A$ as a finite dimensional algebra over $GF(p)$.
    Hence, the algebra $E=\operatorname{End}_{GF(p)}(A)$ of all $GF(p)$-linear transformations on $A$ is isomorphic to a matrix algebra $\mm_n(GF(p))$, where $n=\dim_{GF(p)} A$.
    Let $L \colon A \to E$ be the left regular representation of $A$ given by $L(a)=L_a$ for all $a \in A$, where $L_a$ denotes left multiplication by $a$.
    Now suppose $a \approx_A b$ holds in $A$.
    Since $L$ is a unital ring morphism, item $(ii)$ implies that $L_a \approx_E L_b$.
    The fact that edges in item $(ii)$ must be well defined implies $\CC(L_a)=\CC(L_b)$.
    Since $E$ is isomorphic to a full matrix algebra, it follows from the Centralizer Theorem \cite[p. 113, Corollary 1]{Ja53} that this is equivalent to  $GF(p)[L_a]=GF(p)[L_b]$, where $GF(p)[L_a]$ denotes the unital $GF(p)$-algebra generated by $L_a$, see \cite[Lemma 2.4]{DoGuKuOb13}.
    But $GF(p)\cong \Z_p$ is a factor ring of $\Z$, so that $GF(p)[L_a]=\rng{L_a}_1$ and consequently $\rng{L_a}_1=\rng{L_b}_1$.
    Hence, there exist polynomials $P,Q \in \Z[x]$ such that $L_a=P(L_b)=L_{P(b)}$ and $L_b=Q(L_a)=L_{Q(a)}$. Applying these transformations to $1 \in A$ gives $a=P(b)$ and $b=Q(a)$, hence $\rng{a}_1=\rng{b}_1$ and $a \sim_1 b$.
    \end{proof}

\begin{remark}\label{rem:best}
    Note that the nonunital version of Theorem~\ref{thm:best1} also holds. If $A$ is any finite algebra over a field $\F$ with prime field $GF(p)$, then $A^1$ is a finite unital $GF(p)$-algebra and we have an inclusion $i \colon A \to A^1$.
    If $a \approx_A b$ then $i(a) \approx_{A^1} i(b)$ and the same proof as above applied to $A^1$ show that 
    $i(a) \sim_1 i(b)$, i.e. $[i(a)]_1=[i(b)]_1$. By Proposition~\ref{prop:iso} we conclude that $[a]=[b]$ and $a \sim b$. 
\end{remark}

Next we show that functors $\Lambda$ and $\Lambda^1$ preserve embeddings, i.e., injective morphisms.

\begin{lemma}\label{lem:embedding}
    If $f \colon R \to S$ is a ring embedding, then $\Lambda(f) \colon \Lambda(R) \to \Lambda(S)$ is a graph embedding. If, in addition, $R$, $S$, and $f$ are unital, then $\Lambda^1(f) \colon \Lambda^1(R) \to \Lambda^1(S)$ is a graph embedding as well.
\end{lemma}

\begin{proof}
    Let $f \colon R \to S$ be a unital ring embedding. We only need to prove that $\Lambda^1(f)$ is injective.
    If $[f(a)]_1=[f(b)]_1$ then there exist polynomials $p,q \in \Z[x]$ such that $f(a)=p(f(b))$ and $f(b)=q(f(a))$.
    Since $f$ is a unital ring morphism we infer $f(a-p(b))=f(b-q(a))=0$ and by the injectivity of $f$ we conclude that $a=p(b)$ and $b=q(a)$.
    This means that $[a]_1=[b]_1$, so that $\Lambda^1(f)$ is injective.
    The proof for $\Lambda$ is similar.
\end{proof}

If $G$ and $H$ are two graphs we denote by $G \cup H$ their disjoint union, by $tG$ a union of $t$ copies of $G$ and by $G \vee H$ their join, i.e. the graph with  $V(G \vee H) = V(G) \cup V(H)$ and $E(G \vee H) = E(G) \cup E(H) \cup \{\{a,b\} ~|~ a \in V(G), b \in V(H)\}$. 
The tensor product $G\times H$ of graphs $G$ and $H$ is defined as follows.
The set of vertices of $G\times H$ is the Cartesian product $V\left(G\right)\times V\left(H\right)$
and there is an edge between $\left(g,h\right)$ and $\left(g',h'\right)$
if and only if there is an edge between $g$ and $g'$ in $G$, and an edge between $h$ and $h'$
in $H$.
We denote by $K_n$ the complete graph on $n$ vertices without any loops and by $K^\circ_n$ the complete graph on $n$ vertices with all the loops.

We introduce the notation for the number of vertices of (unital) compressed commuting graph.
\begin{definition}
Let $R$ be a ring. We denote the number of vertices of $\Lambda^1(R)$ and $\Lambda(R)$ by $v_1(R)$ and $v(R)$, respectively.
\end{definition}

It is clear that if $R$ is a commutative ring then $\Lambda^1(R)\cong K_{v_1(R)}^\circ$ and $\Lambda(R)\cong K_{v(R)}^\circ$. 

Let $R$ be a (unital) ring and $\Gamma(R)$ its (usual) commuting graph, i.e. $V(\Gamma(R))$ is the set of all non-central elements of $R$, and $a,b \in V(\Gamma(R))$ are connected if and only if $a\ne b$ and $ab=ba$. If $a,b \in V(\Gamma(R))$, $[a]_1 = [b]_1$, and $a \ne b$, then $a$ and $b$ are connected in $\Gamma(R)$. 
Furthermore, if $a,b,c,d \in V(\Gamma(R))$ and $[a]_1 = [b]_1 \ne [c]_1 = [d]_1$, then $a$ and $c$ are connected in $\Gamma(R)$ if and only if $b$ and $d$ are connected in $\Gamma(R)$. So the (unital) compressed commuting graph of $R$ is obtained from $\Gamma(R)$ by compressing some cliques in $\Gamma(R)$ into single points, adding all the loops and taking the join with the (unital) compressed commuting graph of the center of $R$. 

We will consider also weighted (unital) compressed commuting graphs of finite (unital) rings. A \emph{weighted graph} is a graph $G$ together with a function $w: V(G) \to \RR$ of weights of vertices. We denote by $K^\circ_n(w_1, w_2, \ldots, w_n)$ the weighted complete graph with all the loops, where the weights of the vertices are $w_1, w_2, \ldots, w_n$. If we permute the weights $w_1, w_2, \ldots, w_n$, we obtain an isomorphic weighted graph, so we will usually assume that $w_1 \ge w_2 \ge \ldots \ge w_n$.

\begin{definition}\label{def:weighted}
If $R$ is a finite ring, its \emph{weighted compressed commuting graph} $\Lambda_w(R)$ is the graph $\Lambda(R)$ with weight function $w([a]) = |[a]|$. If $R$ is a finite unital ring, its \emph{weighted unital compressed commuting graph} $\Lambda^1_w(R)$ is the graph $\Lambda^1(R)$ with weight function $w([a]_1) = |[a]_1|$.    
\end{definition}

The commuting graph $\Gamma(R)$ can be reconstructed from (unital) weighted compressed commuting graph $\Lambda_w(R)$ or $\Lambda^1_w(R)$. Namely, $\Gamma(R)$ can be obtained from $\Lambda_w(R)$ or $\Lambda^1_w(R)$ by first removing all vertices, which are connected to every vertex (and correspond to central elements), then blowing-up every vertex $v$ of the (unital) weighted compressed commuting graph into $w(v)$ connected copies and finally removing all loops.

While it is not clear which of the two graphs $\Gamma(R)$ or $\Lambda(R)$ contains more information about the ring $R$, it is certainly the case that the graph $\Lambda_\omega(R)$ contains the most information among these three graphs.

\section{Finite fields and related rings}\label{sec:fields}

Let $p$ be a prime and $n \ge 1$ an integer. We denote by $GF(p^n)$ the finite field with $p^n$ elements. In this section we describe the (unital) compressed commuting graph of the field $GF(p^n)$, the graph of a direct product of two arbitrary finite fields, and the graph of a specific factor ring of the polynomial ring.

From now on we are going to regard the vertex set of the (unital) compressed commuting graph of $R$ to be the set of all (unital) subrings of $R$ generated by one element, see Remark~\ref{rem:def}.

We denote by $d(n)$ the number of all positive divisors of $n$ and by $\sigma(n) = \sum_{d | n}d$ the Euler sigma function, i.e., the sum of all positive divisors of $n$.

\begin{theorem}\label{th:field}
Let $p$ be a prime and $n \ge 1$ an integer. Then the (unital) compressed commuting graph of the field $GF(p^n)$ is a complete graph on $d(n)$ respectively $d(n)+1$
vertices with all the loops, i.e.
$$ \Lambda^1(GF(p^n)) \cong K_{d(n)}^\circ \quad\text{ and }\quad \Lambda(GF(p^n)) \cong K_{d(n)+1}^\circ. $$
\end{theorem}

\begin{proof}
Since a field is commutative, its (unital) compressed commuting graph is a complete graph with all the loops.
If $x\ne 0$, then $\langle x \rangle_1 = \langle x \rangle$ equals the subfield generated by $x$.
Furthermore, $\langle 0 \rangle_1 = GF(p)$ and $\langle 0 \rangle = \{0\}$.
The field $GF(p^n)$ contains a unique subfield isomorphic to $GF(p^m)$ if and only if $m$ is a divisor of $n$, see \cite[Theorem 2.6]{LiNi94}.  So the number of subfields in $GF(p^n)$ equals the number of divisors of $n$, i.e. $d(n)$, and each subfield is generated by a single element. The result follows. 
\end{proof}

\begin{example}
Let $GF(p^6)$ be the finite field with $p^6$ elements. Then
$$\Lambda^1(GF(p^6)) \cong K_4^\circ \quad\text{and}\quad \Lambda(GF(p^6)) \cong K_5^\circ.$$
In the case of unital compressed commuting graph the equivalence classes contain $p, p^2-p, p^3-p,$ and $p^6-p^3-p^2+p$ elements, respectively. In particular, the weighted (unital) compressed commuting graph of $GF(2^6)$ is $K_4^\circ(54, 6, 2, 2)$ or $K_5^\circ(54, 6, 2, 1, 1)$.
\end{example}

\begin{corollary}
For any integer $m \ge 1$  there exist a ring $R$ such that its (unital) compressed commuting graph is a complete graph on $m$ vertices $K_m^\circ$ with all the loops. 
\end{corollary}

\begin{proof}
Let $n= 2^{m-1}$  and $R=GF(p^n)$. Then $d(n) = m$, so $\Lambda^1(R) \cong K_m^\circ$ and $\Lambda(R) \cong K_{m+1}^\circ$ by Theorem \ref{th:field}. Furthermore, $\Lambda^1(\{0\}) \cong \Lambda(\{0\}) \cong K_1^\circ$. 
\end{proof}

The functors $\Lambda^1$ and $\Lambda$ do not preserve products in general. Nevertheless, when the cardinalities of finite rings are relatively prime, the graph of the direct product of these rings is isomorphic to the tensor product of the corresponding graphs.  

\begin{proposition}\label{prop:direct}
Let $R$ and $S$ be two finite rings, such that $\gcd(|R|,|S|)=1$. If $(a,b)\in R\times S$ then $\rng{(a,b)}=\rng{a}\times\rng{b}$ and $\rng{(a,b)}_1=\rng{a}_1\times\rng{b}_1$. In
particular, $v(R\times S)=v(R)v(S)$ and $v_1(R\times
S)=v_1(R)v_1(S)$. Furthermore, $\Lambda(R \times S) \cong \Lambda(R) \times \Lambda(S)$ and $\Lambda^1(R \times S) \cong \Lambda^1(R) \times \Lambda^1(S)$.
\end{proposition}
\begin{proof}
Let $|R|=m$ and $|S|=n$. Then $mm'+nn'=1$ for some
$m',n'\in\mathbb{Z}$. Clearly,  $\rng{(a,b)}\subseteq\rng{a}\times\rng{b}$ and
$\rng{(a,b)}_1\subseteq\rng{a}_1\times\rng{b}_1$. Let $(c,d)\in \rng{a}_1\times\rng{b}_1$.
Then there exist polynomials $P,Q \in \mathbb{Z}[x]$
such that $c=P(a)$ and $d=Q(b)$. Let $K=nn'P+mm'Q\in
\mathbb{Z}[x]$. Then $K(a)=c$ and $K(b)=d$, so that $(c,d)=K((a,b))\in
\rng{(a,b)}_1$. Hence, $\rng{(a,b)}_1=\rng{a}_1\times\rng{b}_1$.
Note that $(a,b) \in R \times S$ and $(c,d) \in R \times S$ commute if and only if $a$ and $c$ commute in $R$ and $b$ and $d$ commute in $S$. This implies that $\Lambda^1(R \times S) \cong \Lambda^1(R) \times \Lambda^1(S)$.
The proof for the nonunital version is similar.
\end{proof}
\begin{remark}
Let $R$ be a finite ring such that $|R|=p_1^{\alpha_1}p_2^{\alpha_2}\cdots
p_k^{\alpha_k}$, where $p_1,p_2,\ldots,p_k$ are distinct primes. Then it is well known that $R=\prod_{i=1}^k R_i$, where $|R_i|=p_i^{\alpha_i}$, see \cite{Sh30}. So
$v(R)=\prod_{i=1}^k v(R_i)$ and $v_1(R)=\prod_{i=1}^k v_1(R_i)$.
Also, $\Lambda^1(R)\cong \prod_{i=1}^k \Lambda^1(R_i)$ and
$\Lambda(R)\cong \prod_{i=1}^k \Lambda(R_i)$. So it suffices
to compute the (unital) compressed commuting graphs of rings whose cardinality is a power of a prime.
\end{remark}

We now consider the graph of a direct product $R=GF(p^n) \times GF(q^m)$ of two finite fields. In Section~\ref{sec:matrices} we are going to need this graph in the case when the two fields are equal and the general case is just slightly more involved.

\begin{theorem} \label{thm:direct_1}
Let $n, m \ge 1$ be integers, $p, q$ be primes, $s = \gcd(n,m)$ the greatest common divisor of $n$ and $m$ and let $R = GF(p^n) \times GF(q^m)$ be the direct product of fields with $p^n$  and $q^m$ elements. 
Then
\begin{equation} \label{eq:t}
v_1(R) = \begin{cases}
    d(n)d(m); & \text{if } p \ne q, \\
    d(n)d(m) + \sigma(s)-1; & \text{if }  p = q = 2 \text{ and }  s \text{ is even}, \\
    d(n)d(m) + \sigma(s); & \text{otherwise,} 
\end{cases}    
\end{equation}
and the unital compressed commuting graph of the ring $R$ is $\Lambda^1(R) \cong K_{v_1(R)}^\circ$.
\end{theorem}

\begin{proof}
The ring $R$ is commutative, so $\Lambda^1(R) \cong K_{v_1(R)}^\circ$.
If $p \ne q$, the result follows from Theorem~\ref{th:field} and Proposition~\ref{prop:direct}.

Assume now that $p = q$. For positive integer $r$ and any $\alpha \in GF(p^r)$ we denote by $m_\alpha$ the monic minimal polynomial of $\alpha$ over $GF(p)$.
Let $A=(\alpha,\beta) \in R$ be an arbitrary element and denote its monic minimal polynomial over $GF(p)$ by $m_A$. Clearly, $m_A=LCM(m_\alpha,m_\beta)$, the least common multiple of $m_\alpha$ and $m_\beta$.
If $m_\alpha \ne m_\beta$, then $m_A=m_\alpha m_\beta$ since polynomials $m_\alpha$ and $m_\beta$ are irreducible. Hence,
\begin{align*}
    \dim_{GF(p)} \rng{A}_1 &=\deg m_A=\deg m_\alpha + \deg m_\beta = \dim_{GF(p)} \rng{\alpha}_1 +\dim_{GF(p)} \rng{\beta}_1 \\
    &=\dim_{GF(p)} \big(\rng{\alpha}_1 \times \rng{\beta}_1\big).
\end{align*}
Since, $\rng{A}_1 \subseteq \rng{\alpha}_1 \times \rng{\beta}_1$, we conclude that
\begin{equation}\label{eq:oplus}
    \rng{A}_1 = \rng{\alpha}_1 \times \rng{\beta}_1=GF(p^{d_1}) \times GF(p^{d_2}),
\end{equation}
where $d_1=\deg m_\alpha$ and $d_2=\deg m_\beta$.

On the other hand, if $m_\alpha = m_\beta$, let $d=\deg m_\alpha$, so that $\alpha, \beta \in GF(p^d)$. Since also $\alpha \in GF(p^n)$, $d$ divides $n$, and similarly $d$ divides $m$, so that $d$ divides $s=\gcd(n,m)$. There exists an automorphism $\phi$ of $GF(p^d)$ such that $\phi(\alpha)=\beta$.
It follows that
\begin{equation}\label{eq:phi}
    \rng{A}_1=\{(\gamma,\phi(\gamma)) \in R \mid \gamma \in GF(p^d)\}
\end{equation}
since we have for any $P \in \Z[x]$
$$P(A)=(P(\alpha),P(\beta))=(P(\alpha),P(\phi(\alpha)))=(P(\alpha),\phi(P(\alpha))).$$
Every automorphism of the field $GF(p^d)$ is of the form $\phi(x) = x^{p^j}$ for some integer $j$ with $0 \le j \le d-1$, see \cite[Theorem 2.21]{LiNi94}.

Now we count the number of unital subrings of $R$ generated by one element. They are all of the form \eqref{eq:oplus} or \eqref{eq:phi}. For any divisor $d$ of $s$ we have $d$ subrings of the form \eqref{eq:phi}. In total there are $\sum_{d | s} d = \sigma(s)$ such subrings.

Let $d_1$ divide $n$ and $d_2$ divide $m$. If $d_1 \ne d_2$ then the minimal polynomial of any generator of $GF(p^{d_1})$ is different from the minimal polynomial of any generator of $GF(p^{d_2})$.
So the unital ring $GF(p^{d_1}) \times GF(p^{d_2})$ is generated by one element according to equation \eqref{eq:oplus}. Now suppose $d:=d_1=d_2$. There exist only one irreducible polynomial of degree 2 over $GF(2)$, namely, $m(x) = x^2+x+1$. So the ring $GF(2^2) \times GF(2^2)$ with $p=2$ and $d=2$ is not generated by one element. So suppose $p>2$ or $d \ne 2$.
Then, as proved below,  there exist at least two distinct irreducible polynomials of degree $d$ over $GF(p)$, so again the unital ring $GF(p^{d}) \times GF(p^{d})$ is generated by one element. Indeed, the number of monic irreducible polynomials of degree $d$ over $GF(p)$ equals
$$ N_p(d) = \frac1d \sum_{r | d} \mu(r)p^{\frac{d}{r}},$$
where
$$\mu(r) = \begin{cases} 1; & \text{if } r=1, \\
(-1)^z; & \text{if } r \text{ is a product of } z \text{ distinct primes}, \\0; & \text{if } r \text{ is divisible by a square of a prime}\end{cases}$$
is the M\"obius function, see \cite[Theorem 3.25]{LiNi94}. If $d=1$, we have $N_p(d) = p>1$. If $d=2$, we have $N_p(d) = \frac12(p^2-p)$, which is greater then 1 as soon as $p>2$. If $d > 2$, the first term in the sum is $p^d$, and the rest can be estimated from below by $- \sum_{j=1}^{\frac{d}{s}} p^j$, since $\mu(j) \ge -1$, where $s\ge 2$ is the smallest prime divisor of $d$. It follows that
$$ N_p(d) \ge \frac1d \left(p^d - \sum_{j=1}^{\lfloor\frac{d}{2}\rfloor} p^j\right) \ge \frac1d \left(p^d - \left\lfloor\textstyle\frac{d}{2}\right\rfloor p^{\lfloor\frac{d}{2}\rfloor}\right) = \frac{p^{\lfloor\frac{d}{2}\rfloor}}{d}\left(p^{\lceil\frac{d}{2}\rceil} - \left\lfloor\textstyle\frac{d}{2}\right\rfloor\right) =v.$$
It can be checked by hand that if $p=2$ and $d \le 5$ then $v>1$. Let $p=2$ and $d \ge 6$. Then $p^{\lfloor\frac{d}{2}\rfloor} \ge d$ and $p^{\lceil\frac{d}{2}\rceil}>\lfloor\frac{d}{2}\rfloor$, which can be proved by induction separately for odd and even $d$.
If $p \ge 3$ and $d \ge 3$ then also $p^{\lfloor\frac{d}{2}\rfloor} \ge d$ and $p^{\lceil\frac{d}{2}\rceil}>\lfloor\frac{d}{2}\rfloor$, which is again proved by induction with $p$ fixed.
We conclude that in any case $v>1$, so that $N_p(d) \ge 2$.

So, if $s$ is odd or $p>2$, there are $d(n)d(m)$ unital subrings of $R$ generated by one element of the form \eqref{eq:oplus}. If $p=2$ and $s$ is even there are $d(n)d(m)-1$ unital subrings of $R$ generated by one element of the form \eqref{eq:oplus}.

Adding $\sigma(s)$ gives the desired result for the number $v_1(R)$.
\end{proof}

\begin{theorem} \label{thm:direct}
Let $n, m \ge 1$ be integers, $p, q$ be primes, $s = \gcd(n,m)$ the greatest common divisor of $n$ and $m$ and let $R = GF(p^n) \times GF(q^m)$ be the direct product of fields with $p^n$  and $q^m$ elements. 
Then
\begin{equation} \label{eq:t'}
v(R) = \begin{cases}
    (d(n)+1)(d(m)+1); & \text{if } p \ne q, \\
    (d(n)+1)(d(m)+1) + \sigma(s)-2; & \text{if }  p = q = 2 \text{ and }  s \text{ is even}, \\
    (d(n)+1)(d(m)+1) + \sigma(s)-1; & \text{if }  p = q = 2 \text{ and }  s \text{ is odd}, \\
    (d(n)+1)(d(m)+1) + \sigma(s); & \text{otherwise,} 
\end{cases}    
\end{equation}
and the compressed commuting graph of the ring $R$ is $\Lambda(R) \cong K_{v(R)}^\circ$. 
\end{theorem}

\begin{proof}
As in the proof of Theorem~\ref{thm:direct_1} we have $\Lambda(R) \cong K_{v(R)}^\circ$.
If $p \ne q$, the result follows from Theorem~\ref{th:field} and Proposition~\ref{prop:direct}.

Now suppose $p=q$ so that $R$ is an algebra over $GF(p)$. Let $A=(\alpha,\beta) \in R$.
If $\alpha=\beta=0$, then $\rng{A}=0$.
If $\alpha=0$ and $\beta \ne 0$, then $\rng{A}=GF(p^d) \times 0$ where $d$, the degree of the minimal polynomial of $\alpha$ over $GF(p)$, divides $n$.
Similarly, if $\alpha =0$ and $\beta \ne 0$, then $\rng{A}=0 \times GF(p^{d'})$ where $d'$, the degree of the minimal polynomial of $\beta$ over $GF(p)$, divides $m$.
Note that none of the above rings contains the identity element $1$ and there are $d(n)+d(m)+1$ of them.

Now, if $\alpha \ne 0$ and $\beta \ne 0$, then $A$ is invertible in $R$ so its minimal polynomial over $GF(p)$ must have nonzero constant term. Hence, $1 \in \rng{A}$ and $\rng{A}=\rng{A}_1$. The arguments below rely on the proof of Theorem~\ref{thm:direct_1}. Since we already know the number of unital subrings of $R$ generated by one element, we just need to check which ones are generated by one element also as nonunital rings.
We need to consider all the rings $S$ of type \eqref{eq:oplus} and \eqref{eq:phi}. In order to prove that $S$ is generated by one element as a nonunital ring, it is enough to prove that $S$ is generated as a unital ring by an invertible element. If $S$ is of type \eqref{eq:phi}, it is generated by $(\alpha,\phi(\alpha))$, where $\alpha$ is a generator of $GF(p^d)$.
Now let $S=GF(p^{d_1}) \times GF(p^{d_2})$. If $d_1 \ne d_2$, then $S$ is generated by $(\alpha,\beta)$, where $\alpha$ generates $GF(p^{d_1})$ and $\beta$ generates $GF(p^{d_2})$.
So assume $d:=d_1=d_2$. If $d>1$, then necessarily $d \ne 2$ or $p>2$ and there exist at least two distinct irreducible monic polynomials of degree $d$ over $GF(p)$ (see the proof of Theorem~\ref{thm:direct_1}), say $P$ and $Q$. Let $\alpha,\beta \in GF(p^d)$ be the zero of $P$ and $Q$, respectively. Then $S$ is generated by $(\alpha,\beta)$.
If $d=1$ and $p>2$, then $S=GF(p)\times GF(p)$ is generated by $(1,2)$. Finally, if $d=1$ and $p=2$, then $S=GF(2) \times GF(2)$. In this case $S$ is not generated by one element as a nonunital ring, since it has $4$ elements, but any subring generated by one element contains at most $2$ elements.

So in case $p=q$ we have $v(R)=v_1(R)+(d(n)+d(m)+1)$ if $p>2$, and $v(R)=v_1(R)+(d(n)+d(m)+1)-1$ if $p=2$. The conclusion now follows.
\end{proof}

Next we consider the (unital) compressed commuting graph of a specific factor ring of the polynomial ring, which will be used in next section. 

\begin{proposition}\label{thm:poly-unital}
Let $n$ be an integer, $p$ a prime, and $R= GF(p^n)[x]/(x^2)$ the factor ring of the polynomial ring. Then
$$v_1(R)=d(n)+\sum_{d|n} \frac{p^n-1}{p^d-1}$$ and $\Lambda^1(R)\cong
K_{v_1(R)}^\circ$.
\end{proposition}
\begin{proof}
Let $X=x+(x^2)\in R$. Then $R=GF(p^n)+GF(p^n)X$ and the sum is a direct sum of additive groups. Let $C=a+bX\in R$ where $a,b\in GF(p^n)$. It
is clear that $\rng{C}_1=GF(p)[C]$. If $b=0$ then
$C=a\in GF(p^n)$ and $\rng{C}_1$ is a unital subring of $GF(p^n)$. We have $d(n)$ such subrings generated by one element.

Now, let $b\neq 0$ and
$\deg m_a=d$, where $m_a$ is the minimal polynomial of $a$ over $GF(p)$. Note that $d$ divides $n$. Let $g \in \Z[x]$.
Since
$$g(C)=g(a+bX)=g(a)+g'(a)bX,$$
we have $g(C)=0$ if and only if $g(a)=g'(a)=0$.

Let $C=a+bX$ and $C'=a'+b'X$ be two elements of $R$ with $b \ne 0$ and $b'\ne 0$. 
We claim that $\rng{C}_1=\rng{C'}_1$ if and only if $\rng{a}_1=\rng{a'}_1=GF(p^d)$ and $b GF(p^d)=b' GF(p^d)$, i.e., $b$ and $b'$ individually generate the same $1$-dimensional vector subspace of $GF(p^n)$ over $GF(p^d)$.
Assume first that 
$\rng{C}_1=\rng{C'}_1$. Then there exist two
polynomials $P,Q\in\mathbb{Z}[x]$ such that
$P(a)+P'(a)bX=P(C)=C'=a'+b'X$ and
$Q(a')+Q'(a')b'X=Q(C')=C=a+bX$. This implies
$$a'=P(a), \quad b'=P'(a)b, \quad a=Q(a'), \quad \text{and} \quad b=Q'(a')b'.$$
So $\rng{a}_1=\rng{a'}_1=GF(p^d)$ and
$b GF(p^d)=b' GF(p^d)$. Conversely, assume
that $\rng{a}_1=\rng{a'}_1=GF(p^d)$ and $b GF(p^d)=b' GF(p^d)$. Then there exist two
polynomials $P,R\in\mathbb{Z}[x]$ such that $a'=P(a)$ and $b'=R(a)b$ since $a$ generates $GF(p^d)$.
Let $S=P+m_a'(a)^{-1}(R-P')m_a$. Then $S(a)=P(a)=a'$. Furthermore,
$$ S'(a) =P'(a)+m_a'(a)^{-1}(R'(a)-P''(a))m_a(a)+m_a'(a)^{-1}(R(a)-P'(a))m_a'(a)=R(a), $$
so that $S'(a)b=R(a)b=b'$.
Hence, $S(C)=S(a)+S'(a)bX=a'+b'X=C'$, which implies $C' \in \rng{C}_1$.
Similarly, $C \in \rng{C'}_1$, and consequently $\rng{C}_1=\rng{C'}_1$.

Since the number of $1$-dimensional vector subspaces of $GF(p^n)$ over $GF(p^d)$ equals $\frac{p^n-1}{p^d-1}$ it follows that
the number of unital subrings of $R$ of the form $\rng{C}_1$, where $C=a+bX$ with $b \ne 0$, is equal to
$\sum_{d|n} \frac{p^n-1}{p^d-1}$.

We conclude that $v_1(R)=d(n)+\sum_{d|n} \frac{p^n-1}{p^d-1}$ and $\Lambda^1(R)\cong K_{v_1(R)}^\circ$.
\end{proof}

\begin{proposition}\label{thm:poly}
Let $n$ be an integer, $p$ a prime, and $R= GF(p^n)[x]/(x^2)$ the factor ring of the polynomial ring. Then
$$v(R)=1+d(n)+\frac{p^n-1}{p-1}+\sum_{d|n} \frac{p^n-1}{p^d-1}$$ 
and $\Lambda(R)\cong K_{v(R)}^\circ$.
\end{proposition}

\begin{proof}
As in the proof of Proposition~\ref{thm:poly-unital}, let $X=x+(x^2)\in R$ and $C=a+bX\in R$ where $a,b\in GF(p^n)$. 
If $b=0$ then $C=a\in GF(p^n)$ and $\rng{C}$ is a subring of $GF(p^n)$. We have $d(n)+1$ such subrings generated by one element.

Let $b\neq 0$. If $C$ is invertible in $R$, then the ring $\rng{C}$ is unital and 
by Proposition~\ref{thm:poly-unital} we have
$\sum_{d|n} \frac{p^n-1}{p^d-1}$ such subrings, each of them is generated by an invertible element. So, let $C$ be non-invertible.
Then $a=0$ and the minimal polynomial of $C$ over $GF(p)$ is $m_C(x)=x^2$. Furthermore, we have $GF(p)bX = \rng{bX} = \rng{b'X} = GF(p)b'X$
if and only if $b GF(p)=b' GF(p)$, i.e., $b$ and $b'$ individually generate the same $1$-dimensional vector subspace of $GF(p^n)$ over $GF(p)$.
So we have
$\frac{p^n-1}{p-1}$ such subrings. 
We conclude that $v(R)=1+d(n)+\frac{p^n-1}{p-1}+\sum_{d|n} \frac{p^n-1}{p^d-1}$ and $\Lambda(R)\cong K_{v(R)}^\circ.$  
\end{proof}

\section{Rings of $2 \times 2$ matrices over finite fields}\label{sec:matrices}

In this section we describe (unital) compressed commuting graphs of rings of $2 \times 2$ matrices over finite fields.

Let $\F$ be a field and denote by $\mm_2(\F)$ the ring of $2\times 2$ matrices over $\F$. It is well-known (see \cite{AkGhHaMo04}) that the commuting graph of $\mm_2(\F)$ is not connected. 
It is composed of cliques containing all non-scalar matrices of the form $aA + bI$, where $A \in \mm_2(\F)$ is a nonscalar matrix and $a, b \in \F$. If $\F$ is the finite field $GF(p^n)$ then $\mm_2(GF(p^n))$ contains $p^{4n}$ elements, and its commuting graph $\Gamma(\mm_2(GF(p^n)))$ has $p^{4n}-p^n$ vertices.
Each clique of $\Gamma(\mm_2(GF(p^n)))$ contains $p^n(p^n-1)=p^{2n}-p^n$ vertices, thus
$$\Gamma(\mm_2(GF(p^n))) = (p^{2n} + p^n + 1)K_{p^{2n}-p^n} .$$

\begin{theorem} \label{th:M2-unital}
Let $n$ be an integer, $p$ a prime, and $GF(p^n)$ the field with $p^n$ elements. Let 
\begin{align*}
   a(n) &= \begin{cases}
    d(n)^2 - d(n) + \sigma(n)-1; & \text{if } p =2 \text{ and } n \text{ is even}, \\
    d(n)^2 - d(n) + \sigma(n); & \text{if } p >2 \text{ or } n \text{ is odd}, 
    \end{cases}\\
   b(n) &= \sum_{d | n} \frac{p^n-1}{p^d-1}, \quad \text{and} \quad  c(n) = d(2n)-d(n).
\end{align*} 
Then the unital compressed commuting graph of the ring $\mm_2(GF(p^n))$ is 
$$ \Lambda^1(\mm_2(GF(p^n))) \cong K_{d(n)}^\circ \vee \left(\tfrac{p^{2n} + p^n}{2}K_{a(n)}^\circ \cup (p^n + 1)K_{b(n)}^\circ \cup \tfrac{p^{2n} - p^n}{2}K_{c(n)}^\circ\right) .$$
\end{theorem}

\begin{proof}
Let $R=\mm_2(GF(p^n))$. Note that the center of $R$ is $Z(R)= GF(p^n)I$, where $I$ is the identity matrix. For any non central element $A \in R$ its centralizer $\CC(A)=GF(p^n)[A]$ is commutative and $\CC(A)\setminus Z(R)$ is a clique in $\Gamma(R)$ of size $p^{2n}-p^n$.
Since $\Gamma(R) = (p^{2n} + p^n + 1)K_{p^{2n}-p^n}$, we need to show that the unital compressed commuting graph of $Z(R)$ is $K_{d(n)}^\circ$ and that some of the cliques of $\Gamma(R)$ are compressed into $K_{a(n)}^\circ$, some into $K_{b(n)}^\circ$, and some into $K_{c(n)}^\circ$.
Since $Z(R)$ is the field $GF(p^n)$ the first claim follows from Theorem~\ref{th:field}.

So assume that $A\in R\setminus Z(R)$. Since we are in $2 \times 2$ matrices, the minimal polynomial of $A$ over $GF(p^n)$ is equal to its characteristic polynomial.
We have three possible types of centralizers for $A$: (A) the minimal polynomial of $A$ has two distinct roots in $GF(p^n)$, (B) the minimal polynomial of $A$ has two equal roots in $GF(p^n)$, (C) the minimal polynomial of $A$ is irreducible over $GF(p^n)$. We now consider each case.

\noindent{\bf Case~(A).}
If the minimal polynomial of $A$ has two distinct roots in $GF(p^n)$, then $A$ is similar to a diagonal matrix. Hence,
$$\CC(A)\cong \mathcal{D}_2( GF(p^n)) :=\big\{
 \begin{bmatrix}
 a& 0\\
 0 & b\\
\end{bmatrix} \in R \mid a,b\in GF(p^n)\big\}\cong GF(p^n)\times GF(p^n),$$
so the clique $\CC(A) \setminus Z(R)$ compresses into $K_{a(n)}^\circ$, where
$$a(n)=v_1(GF(p^n)\times GF(p^n))-v_1(GF(p^n)).$$
By Theorem~\ref{thm:direct_1} it follows that $a(n)=d(n)^2+\sigma(n)-1-d(n)$ if $p=2$ and $n$ is even and $a(n)=d(n)^2+\sigma(n)-d(n)$ otherwise.
Note that Case~(A) happens if and only if $\CC(A)= GF(p^n)[A]$ contains a nontrivial idempotent element. Indeed, if $\CC(A)$ contains a nontrivial idempotent $E$, then $\CC(A)=\CC(E) \cong \mathcal{D}_2(GF(p^n))$. We denote by $E_{ij}$ the matrix having $1$ as $ij$-th entry and $0$'s elsewhere. Since every nontrivial idempotent is similar to $E_{11}$, we have $|GL_2( GF(p^n)){\,:\,}\mathcal{U}(\mathcal{D}_2( GF(p^n)))|$ nontrivial idempotents, where $\mathcal{U}(\mathcal{D}_2(GF(p^n)))$ is the group of units in $\mathcal{D}_2(GF(p^n))$. We have
$$|GL_2(GF(p^n))|=(p^{2n}-1)(p^{2n}-p^n) \quad \text{and} \quad |\mathcal{U}(\mathcal{D}_2(GF(p^n)))|=(p^n-1)^2,$$
so that there are $p^n(p^n+1)$ nontrivial idempotents in $R$.
Every centralizer in Case~(A) contains two nontrivial idempotents, so there are $\frac{p^n(p^n+1)}{2}$ such  centralizers.

\noindent{\bf Case~(B).} If the minimal polynomial of $A$ has two equal roots, then $A$ is similar to a Jordan block. Hence,
$$\CC(A)\cong \mathcal{T}_2( GF(p^n)) :=\big\{
 \begin{bmatrix}
 a& b\\
 0 & a\\
 \end{bmatrix} \in R \mid a,b\in GF(p^n)\big\}\cong GF(p^n)[x]/(x^2),$$
so the clique $\CC(A) \setminus Z(R)$ compresses into $K_{b(n)}^\circ$, where
$$b(n)=v_1(GF(p^n)[x]/(x^2))-v_1(GF(p^n)).$$
By Theorem~\ref{thm:poly-unital} it follows that 
$b(n)=\sum_{d|n} \frac{p^n-1}{p^d-1}.$
Note that Case~(B) happens if and only if $\CC(A)= GF(p^n)[A]$ contains a nonzero nilpotent element. Since every nonzero nilpotent element is similar to
$E_{12}$, we have
$$|GL_2(GF(p^n)){\,:\,}\mathcal{U}(\mathcal{T}_2( GF(p^n)))|=\frac{(p^{2n}-1)(p^{2n}-p^n)}{p^n(p^n-1)}=p^{2n}-1$$
nonzero nilpotent elements. Every centralizer in Case~(B) contains $p^n-1$ nonzero nilpotent elements, so there are $p^n+1$ such centralizers.

\noindent{\bf Case~(C).} If the minimal polynomial $m_A(x)=x^2+\alpha x+\beta$ of $A$ is irreducible over $GF(p^n)$ then $A$ is similar to the companion matrix of $m_A$, i.e.
$\begin{bmatrix}
 0& -\beta\\
 1& -\alpha\\
\end{bmatrix}$.
Furthermore,
$$\CC(A)= GF(p^n)[A] \cong GF(p^n)[x]/(m_A(x))
\cong GF(p^{2n}),$$
since the ideal $(m_A(x))$ is prime in $GF(p^n)[x]$, so that $GF(p^n)[x]/(m_A(x))$ is a field.
Hence, the clique $\CC(A) \setminus Z(R)$ compresses into $K_{c(n)}^\circ$, where 
$$c(n)=v_1(GF(p^{2n}))-v_1(GF(p^n))=d(2n)-d(n).$$
Note that Case~(C) happens if and only if $\CC(A)= GF(p^n)[A]$ contains a companion matrix of some irreducible polynomial of degree $2$. Indeed, if
$A=\begin{bmatrix}
   a & b\\c & d 
\end{bmatrix},$
then $c \ne 0$ since $m_A$ is irreducible, hence, $A'=c^{-1}(A-aI) \in \CC(A)$ is the companion matrix whose minimal polynomial is automatically irreducible since $\CC(A')=\CC(A)$. 
Furthermore, if matrix $B$ is the companion matrix of polynomial $P(x)=x^2+\alpha x+\beta$, then 
$$\CC(B)= GF(p^n)[B] = \big\{\begin{bmatrix}
  a& -b\beta\\
  b & a-b\alpha\\
   \end{bmatrix} \in R \mid a,b\in GF(p^n)\big\},$$
   which contains only one companion matrix. 
It follows that every centralizer in Case (C) contains precisely one companion matrix of an irreducible polynomial, so there are $\frac{p^n(p^n-1)}{2}$ such centralizers. This finishes the proof. 
\end{proof}

\begin{theorem}
Let $n$ be an integer, $p$ a prime, and $GF(p^n)$ the field with $p^n$ elements. Let 
\begin{align*}
   a'(n) &= \begin{cases}
    d(n)^2 + d(n) + \sigma(n)-2; & \text{if } p =2 \text{ and } n \text{ is even}, \\
    d(n)^2 + d(n) + \sigma(n)-1; & \text{if } p =2 \text{ and } n \text{ is odd},  \\
    d(n)^2 + d(n) + \sigma(n); & \text{if } p >2,     
    \end{cases}\\
   b'(n) &= \frac{p^n-1}{p-1}+\sum_{d | n} \frac{p^n-1}{p^d-1}, \quad \text{ and } \quad c(n) = d(2n)-d(n).
\end{align*} 
Then the unital compressed commuting graph of the ring $\mm_2(GF(p^n))$ is 
$$ \Lambda(\mm_2(GF(p^n))) \cong K_{d(n)+1}^\circ \vee \left(\tfrac{p^{2n} + p^n}{2}K_{a'(n)}^\circ {\textstyle\cup} (p^n + 1)K_{b'(n)}^\circ {\textstyle\cup} \tfrac{p^{2n} - p^n}{2}K_{c(n)}^\circ\right) .$$
\end{theorem}

\begin{proof}
The arguments are the same as in the proof of Theorem \ref{th:M2-unital}. The only difference is in the sizes of complete graphs. First of all, $v(Z(R)) = v(GF(p^n)) = d(n) + 1$.
Furthermore, we use Theorem~\ref{thm:direct} to obtain
\begin{align*}
  a'(n)&=v(GF(p^n)\times GF(p^n))-v(GF(p^n)) \\
  &=\begin{cases}
    d(n)^2 + d(n) + \sigma(n)-2; & \text{if } p =2 \text{ and } n \text{ is even}, \\
    d(n)^2 + d(n) + \sigma(n)-1; & \text{if } p =2 \text{ and } n \text{ is odd},  \\
    d(n)^2 + d(n) + \sigma(n); & \text{if } p >2,     
    \end{cases}  
\end{align*}
we use Theorem~\ref{thm:poly} to obtain
$$b'(n)=v(GF(p^n)[x]/(x^2))-v(GF(p^n))=\frac{p^n-1}{p-1}+\sum_{d | n} \frac{p^n-1}{p^d-1},$$
and we have
\begin{equation*}
    c'(n)=v(GF(p^{2n}))-v(GF(p^n))=d(2n)-d(n)=c(n). 
\end{equation*}
\end{proof}

As already mentioned, for any (unital) ring $R$, the vertex set of its (unital) compressed commuting graph  is the set of (unital) subrings of $R$ generated by one element. So the number of vertices $v(R)$ (resp. $v_1(R)$) in  $\Lambda(R)$ (resp. $\Lambda^1(R)$) equals the number of (unital) subrings of $R$ generated by one element. The following corollary gives the number of (unital) subrings generated by one element of the ring of $2 \times 2$ matrices over small finite fields. 

\begin{table}[ht]
\begin{tabular}{|c||c|c|} \hline
    $\F$ & $v(R)$ & $v_1(R)$ \\ \hline \hline
    $GF(2)$ & $15$ & $8$ \\ \hline
    $GF(p), p > 2$ & $2p^2+3p+4$ & $p^2+p+2$ \\ \hline
    $GF(2^2)$ & $114$ & $68$ \\ \hline
    $GF(p^2), p>2$ & $5p^4 + 2p^3+ 7p^2 + 2p + 6$ & $3p^4 + p^3+ 4p^2 + p + 4$ \\ \hline
\end{tabular}\vspace{3mm}
\caption{Sizes of (unital) compressed commuting graphs of $R=\mm_2(\F)$ for small fields $\F$.} \label{table1}
\end{table}

\begin{corollary}
Let $\F$ be a small finite field and $R = \mm_2(\F)$. Table~\ref{table1} gives the number of (unital) subrings of $R$ generated by one element.
\end{corollary}

Our last corollary describes asymptotic behavior of the number of (unital) subrings of $\mm_2(GF(p^n))$ generated by one element for large $n$.

\begin{corollary}
Let $\F=GF(p^n)$ be a finite field and $R = \mm_2(\F)$. Then for every prime $p$ the number of (unital) subrings of $R$ generated by one element satisfies
$$\lim_{n \to \infty}\frac{v(R)}{\tfrac12\sigma(n)|\F|^2} = \lim_{n \to \infty}\frac{v_1(R)}{\tfrac12\sigma(n)|\F|^2} = 1 ,$$
and thus
$$\limsup_{n \to \infty}\frac{v(R)}{p^{2n} n \log \log n} = \limsup_{n \to \infty}\frac{v_1(R)}{p^{2n} n \log \log n} = \tfrac12 e^{\gamma} ,$$
where $\gamma \approx 0.5772$ is Euler's gamma constant.
\end{corollary}

\begin{proof}
Fix a prime number $p$. It follows from Theorem~\ref{th:M2-unital} that
$$v_1(R)=d(n)+\tfrac{p^{2n} + p^n}{2} a(n) + (p^n + 1)b(n) + \tfrac{p^{2n} - p^n}{2}c(n).$$
We will first prove that
\begin{equation}\label{eq:d2}
    \lim_{n \to \infty} \frac{d(n)^2}{\sigma(n)}=0.
\end{equation}
By \cite[Theorem~317]{HaWr08} we have
$$\limsup_{n \to \infty} \frac{\log d(n) \cdot \log\log n}{\log n}=\log 2,$$
so for $n$ sufficiently large we get
$$\frac{\log d(n) \cdot \log\log n}{\log n}<1.$$
Since $\sigma(n) > n$, we obtain
$$\frac{\log d(n)^2}{\log\sigma(n)}<\frac{2\log n}{\log\sigma(n) \cdot \log\log n} < \frac{2}{\log\log n}<\frac12,$$
for $n$ sufficiently large. This implies $d(n)^2 <\sqrt{\sigma(n)}$ so that \eqref{eq:d2} follows. Equation \eqref{eq:d2} immediately implies
$$\lim_{n \to \infty}\frac{a(n)}{\sigma(n)}=1 \qquad \text{and} \qquad \lim_{n \to \infty}\frac{c(n)}{\sigma(n)}=0,$$
since $c(n) \le d(n)$.
Furthermore, $b(n) \le d(n) \cdot p^n$.
Combining all the above, we infer
$$\lim_{n \to \infty} \frac{v_1(R)}{\tfrac12\sigma(n)|\F|^2} =\lim_{n \to \infty} \left(\frac{d(n)}{\tfrac12\sigma(n)p^{2n}}+(1+\tfrac{1}{p^n})\frac{a(n)}{\sigma(n)}+2(1+\tfrac{1}{p^n})\frac{b(n)}{p^n\sigma(n)}+(1-\tfrac{1}{p^n})\frac{c(n)}{\sigma(n)}\right)=1.$$
By \cite[Theorem~323]{HaWr08} we have
$$\limsup_{n \to \infty} \frac{\sigma(n)}{n \log\log n}=e^\gamma$$
so that
$$\limsup_{n \to \infty}\frac{v_1(R)}{p^{2n} n \log \log n} =\limsup_{n \to \infty} \frac{v_1(R)}{\tfrac12\sigma(n)p^{2n}} \cdot \frac{\tfrac12\sigma(n)}{n \log \log n} = \tfrac12 e^{\gamma}.$$
The proof for $v(R)$ is similar.
\end{proof}

\bibliographystyle{amsplain}

\end{document}